\documentclass[11pt]{amsart}
\usepackage{amssymb,amsthm,amsmath,epsfig,latexsym}
\usepackage[pagebackref]{hyperref}
\usepackage{color}
\usepackage[all]{xy}
\usepackage{calc,times,verbatim}
\usepackage{geometry} 
\begin{document}

\newcommand{\mmbox}[1]{\mbox{${#1}$}}
\newcommand{\proj}[1]{\mmbox{{\mathbb P}^{#1}}}
\newcommand{\Cr}{C^r(\Delta)}
\newcommand{\CR}{C^r(\hat\Delta)}
\newcommand{\affine}[1]{\mmbox{{\mathbb A}^{#1}}}
\newcommand{\Ann}[1]{\mmbox{{\rm Ann}({#1})}}
\newcommand{\caps}[3]{\mmbox{{#1}_{#2} \cap \ldots \cap {#1}_{#3}}}
\newcommand{\Proj}{{\mathbb P}}
\newcommand{\K}{{\mathbb K}}
\newcommand{\N}{{\mathbb N}}
\newcommand{\Z}{{\mathbb Z}}
\newcommand{\R}{{\mathbb R}}
\newcommand{\A}{{\mathcal{A}}}
\newcommand{\Tor}{\mathop{\rm Tor}\nolimits}
\newcommand{\Ext}{\mathop{\rm Ext}\nolimits}
\newcommand{\Hom}{\mathop{\rm Hom}\nolimits}
\newcommand{\im}{\mathop{\rm Im}\nolimits}
\newcommand{\rank}{\mathop{\rm rank}\nolimits}
\newcommand{\supp}{\mathop{\rm supp}\nolimits}
\newcommand{\arrow}[1]{\stackrel{#1}{\longrightarrow}}
\newcommand{\CB}{Cayley-Bacharach}
\newcommand{\coker}{\mathop{\rm coker}\nolimits}
\sloppy
\newtheorem{defn0}{Definition}[section]
\newtheorem{prop0}[defn0]{Proposition}
\newtheorem{quest0}[defn0]{Question}
\newtheorem{thm0}[defn0]{Theorem}
\newtheorem{lem0}[defn0]{Lemma}
\newtheorem{corollary0}[defn0]{Corollary}
\newtheorem{example0}[defn0]{Example}
\newtheorem{remark0}[defn0]{Remark}
\newtheorem{conj0}[defn0]{Conjecture}
\newtheorem{notation0}[defn0]{Notation}

\newenvironment{defn}{\begin{defn0}}{\end{defn0}}
\newenvironment{prop}{\begin{prop0}}{\end{prop0}}
\newenvironment{quest}{\begin{quest0}}{\end{quest0}}
\newenvironment{thm}{\begin{thm0}}{\end{thm0}}
\newenvironment{lem}{\begin{lem0}}{\end{lem0}}
\newenvironment{cor}{\begin{corollary0}}{\end{corollary0}}
\newenvironment{exm}{\begin{example0}\rm}{\end{example0}}
\newenvironment{rem}{\begin{remark0}\rm}{\end{remark0}}

\newcommand{\defref}[1]{Definition~\ref{#1}}
\newcommand{\propref}[1]{Proposition~\ref{#1}}
\newcommand{\thmref}[1]{Theorem~\ref{#1}}
\newcommand{\lemref}[1]{Lemma~\ref{#1}}
\newcommand{\corref}[1]{Corollary~\ref{#1}}
\newcommand{\exref}[1]{Example~\ref{#1}}
\newcommand{\secref}[1]{Section~\ref{#1}}
\newcommand{\remref}[1]{Remark~\ref{#1}}
\newcommand{\questref}[1]{Question~\ref{#1}}

\newcommand{\std}{Gr\"{o}bner}
\newcommand{\jq}{J_{Q}}

\def\demo{\noindent {\bf Proof.}\;}

\newcommand{\rar}{\rightarrow}
\newcommand{\lar}{\longrightarrow}
\newcommand{\llar}{-\kern-5pt-\kern-5pt\longrightarrow}
\newcommand{\surjects}{\twoheadrightarrow}
\newcommand{\injects}{\hookrightarrow}

\newcommand{\Fiber}{{\cal F}}

\newcommand{\U}{\mbox{U}}
\renewcommand{\P}{\mbox{P}}
\def\fm{{\mathfrak m}}
\def\fn{{\mathfrak n}}
\def\ii{\'{\i}}


\def\Ree#1{{\cal R}(#1)}
\def\gr{\rm gr}
\def\Ht{{\rm ht}\,}
\def\depth{{\rm depth}\,}
\def\ass{{\rm Ass}\,}
\def\codim{{\rm codim}\,}
\def\edim{{\rm edim}\,}
\def\ecodim{{\rm ecodim}\,}
\def\ker{{\rm ker}\,}
\def\grade{{\rm grade}\,}
\def\rk{{\rm rank}\,}
\def\syz{\mbox{\rm Syz}}
\def\sym#1{\mbox{\rm Sym}(#1)}
\def\spec#1{{\rm Spec}\, (#1)}
\def\supp#1{{\rm Supp}\, (#1)}
\def\br#1#2{{\rm Branch}\, (#1/#2)}
\def\bdel{{\bf\Delta}}
\def\bd{{\bf D}}
\def\bfd{{\bf D}}
\def\grD{\sum_{j\geq 0} \dd^j/\dd^{j+1}}
\def\ne{{\mathbb E}}
\def\nf{{\mathbb F}}
\def\sse{\Leftrightarrow}
\def\ini{\mbox{\rm in}}
\def\join{{\mathfrak J}(I,J)}
\def\spanal{{\mathbb S}\mbox{\rm p}}
\def\der{\mbox{\rm Der}}
\def\dera#1{\mbox{\rm Der}_k(#1)}
\def\highder#1#2{\mbox{\rm Der}_k^{\,#1}(#2)}
\def\ldera#1{\mbox{\rm Lder}_k(#1)}
\def\derab#1#2{\mbox{\rm Der}_k(#1,#2)}
\def\highdif#1#2{\mbox{\rm Diff}_k^{\,#1}(#2)}
\def\I{{\cal I}}
\def\O{{\cal O}}
\def\restr{{\kern-1pt\restriction\kern-1pt}}
\def\ldasharrow{{--\dasharrow}}

\def\Q{{\mathbb Q}}
\def\R{{\mathbb R}}
\def\C{{\mathbb C}}
\def\Z{{\mathbb Z}}
\def\NN{\mathbb N}
\def\dd{{\mathbb D}}
\def\za{{\mathbb S}_{A/k}}
\def\sta{{\mathbb T}_{A/k}}
\def\cd#1#2{{\mathfrak c}(#1/#2)}
\def\pp{{\mathbb P}}
\def\GG{{\mathbb G}}
\def\XXX{{\mathbb X}}
\def\YYY{{\mathbb Y}}
\def\TTT{{\mathbb T}}


\title{On the Jacobian ideal of central arrangements}

\author{Ricardo Burity, Aron Simis and \c{S}tefan O. Toh\v{a}neanu}

	 \keywords{Jacobian ideal, arrangement of hypersurfaces, depth, reduction of an ideal. \\
		\indent The first author was partially supported by CAPES  (Brazil) (grant: PVEX - 88881.336678/2019-01). \\
		\indent The second author was partially supported by a grant from CNPq (Brazil) (302298/2014-2).\\
		\indent Burity's address: Departamento de Matemática, Universidade Federal da Paraiba, J. Pessoa, Paraiba, 58051-900, Brazil,
		Email: ricardo@mat.ufpb.br.\\
        \indent Simis' address: Departamento de Matemática, Universidade Federal da Pernambuco, Recife, Pernambuco, 50740-560, Brazil,
		Email: aron@dmat.ufpe.br.\\
		\indent Tohaneanu's address: Department of Mathematics, University of Idaho, Moscow, Idaho 83844-1103, USA, Email: tohaneanu@uidaho.edu.}

\begin{abstract}
Let $\mathcal{A}$ denote a  central hyperplane arrangement of rank $n$ in affine space $\mathbb{K}^n$ over an infinite field $\mathbb{K}$ and let $l_1,\ldots, l_m\in R:= \K[x_1,\ldots,x_n]$ denote the linear forms defining the corresponding hyperplanes, along with the corresponding defining polynomial $f:=l_1\cdots l_m\in R$.
Let $J_f$ denote the ideal generated by the partial derivatives of $f$ and let $\mathbb{I}$ designate the ideal generated by the $(m-1)$-fold products of $l_1,\ldots, l_m$.
This paper is centered on the relationship between the two ideals $J_f, \mathbb{I}\subset R$, their properties and two conjectures related to them.
Some parallel results are obtained in the case of forms of higher degrees provided they fulfill a certain transversality requirement.

\end{abstract}

\maketitle

\section*{Introduction}

Let $\mathcal{A}$ denote a  central hyperplane arrangement of rank $n$ in affine space $\mathbb{K}^n$ over an infinite field $\mathbb{K}$ and let $l_1,\ldots, l_m\in R:=\mathbb K[x_1,\ldots,x_n]$ denote the linear forms defining the corresponding hyperplanes.
Set $f:=l_1\cdots l_m\in R$, the defining polynomial of $\mathcal{A}$. The module of {\em logarithmic derivations} associated to $f$ is defined as $\mathrm{Derlog}(\mathcal{A}):=\{\theta\in {\rm Der}(R)| \theta (f)\subset \langle f \rangle\}.$ Its $R$-module structure  is well known if ${\rm char}(\mathbb K)$ does not divide $m=\mathrm{deg}(f)$:
\begin{equation*}
\mathrm{Derlog}(\mathcal{A})= \mathrm{Syz}(J_f)\oplus R\theta_E.
\end{equation*}
Here $\mathrm{Syz}(J_f)\subset R^n$ is up to the identification ${\rm Der}(R)=R^n$ the syzygy module of the  partial derivatives of the $f$ and $\theta_E=\sum_{i=1}^{n}x_i\frac{\partial}{\partial{x_i}}$ is the Euler derivation.

In the case where the arrangement is generic, -- meaning that every subset of $\{l_1,\ldots,l_m\}$ with $n$ elements is $\K$-linearly independent -- and that $m\geq n+1$, Rose-Terao (\cite{RoseTerao}) and Yuzvinsky (\cite{Yuzvinsky}) have established that the homological dimension of $\mathrm{Derlog}(\mathcal{A})$ is $n-2$.   The statement is equivalent to having  ${\rm depth}(R/J_f)=0$, or still that the irrelevant maximal ideal $\fm:=\langle x_1,\ldots,x_n\rangle$ is an associated prime of $R/J_f$.
In their beautiful paper,  Rose and Terao took the approach of  establishing explicit free resolutions of the modules $\Omega^q(\mathcal{A})$  of logarithmic $q$-forms, for $0\leq q\leq n-1$, in the spirit of Lebelt's work (\cite{Lebelt}). Then, drawing on the isomorphism $\mathrm{Derlog}(\mathcal{A})\simeq \Omega^{n-1}(\mathcal{A})$, they derived the free resolution of the module of logarithmic derivations, and hence of $R/J_f$ as well.

Motivated by this, the present authors consider the problem  under a more intrinsic perspective, trying to envisage what can be said, avoiding details of free resolutions,  by drawing upon a few well established facts coming from commutative algebra.
As such, it seemed natural to ask if there is any result of this sort in the case of forms $f_1,\ldots,f_m\in R, m\geq 2$, of degrees $\deg(f_i)=d_i\geq 2, i=1,\ldots,m$.
Accordingly, the paper is focused on two major aspects: first, the nature of $J_f$ from the ideal theoretic and homological point of view, for a central generic hyperplane arrangement; second, the nature of $J_f$ for arrangements of forms that are isolated singularities.

A full grasp of the algebraic properties of a central hyperplane arrangement does not seem to be available.
As is well-known to experts, even in the plane case a question as to how the combinatorial side  affects the homological properties of the Jacobian ideal of the defining polynomial of the arrangement is not entirely settled. The best example is Terao's conjecture which states that, over a field of characteristic zero, freeness of $\mathrm{Derlog}(\mathcal{A})$ depends only on the combinatorics.


For a  central hyperplane arrangement $\mathcal{A}$ of rank $n$ over an infinite field $\mathbb{K}$ one introduces the ideal $\mathbb I\subset R$ generated by the $(m-1)-$fold products of the defining linear forms of $\mathcal{A}$.
While working on the above problem, we came up with the following conjecture, previously unknown at least to us:

{\sc Conjecture 1.} {\rm (${\rm char}(\K)\nmid m$)}\;
Assume that $\mathcal{A}$ is a central hyperplane arrangement of rank $n$ and size $m$. Let $\mathbb{I}\subset R=\K[x_1,\ldots,x_n]$ denote the ideal generated by the $(m-1)$-fold products of $l_1,\ldots,l_m$. Then the ideal $J_f\subset R$ generated by the partial derivatives of $f$ is a minimal reduction of $\mathbb{I}$.

Since $\mathcal{A}$ has rank $n$, one has $m\geq n$. The case where $m=n$ is clear because $J_f=\mathbb{I}$ since they have the same number of minimal generators, and the analytic spread of $\mathbb{I}$ is $n$.

We answer affirmatively this conjecture assuming that every subset of $\{l_1,\ldots,l_m\}$ with $n-1$ elements is $\K$-linearly independent -- this includes the generic case and the arbitrary $n=3$ case.

While focusing on this conjecture, we felt naturally compelled to ask about further behavior pattern by the Jacobian ideal of a central hyperplane arrangement.
We thus came up with a second question:

{\sc Conjecture 2.}  {\rm (${\rm char}(\K)\nmid m$)}\;
Let $f\in R=\K[x_1,\ldots,x_n]$ denote the defining polynomial of a central hyperplane arrangement of size  $m\geq n$. Then the Jacobian ideal $J_f$ is of linear type.

The conjecture trivially holds for $n=2$.
In this work we show that this conjecture holds in complete generality for $n= 3$.
In the generic case and arbitrary $n\geq 3$, we show that $J_f$ at least satisfies condion $G_{\infty}$ (also denoted $F_1$), which is the so to say  easier half of the linear type property.

We now describe the main contents of the paper.

Under the assumption that $\mathcal{A}$ is generic and $m\geq n+1$, the first section deals with the ideal theoretic and homological properties of the Jacobian ideal $J_f$.
We take a kind of ad hoc approach, by closely establishing the algebraic properties of $J_f$ and some of the invariants attached to it.
From these, we derive that $\depth R/J_f=0$ and that its minimal free resolution is $2$-step linear with syzygy module equigenerated in degree $m-n+1$.
Among those properties, we give a priori the exact value of the initial degree of $\syz(J_f)$,  the regularity and the saturation exponent of $J_f$.
When $n\geq 3$ we show that the saturation of $J_f$ with respect to $\fm$ is the ideal $\mathbb{I}$ generated by the $(m-1)$-fold products of $f$.

The second section looks more closely at the inclusion $J_f\subset \mathbb{I}$.
We already know that these ideals share the same radical, but wish to go further to obtain that this inclusion is  a minimal reduction situation, which is the content of the first of the above conjectures.
By drawing upon a couple of methods, we verify this conjecture in a first few cases, including the case where $n=2$ and the case where $n=3$ with a coloop.
We also describe an effective partition method.

Unfortunately, none of this is directly applicable to the arbitrary generic case, which requires a separate argument.
This is then the main result of the section.
The proof is a kind of descent argument involving the various types of natural generators of the critical power $\mathbb{I}^n$. The reason one aims at this power is because $n-1$ is an upper bound for the reduction number of $\mathbb{I}$ (\cite[Corollary 2.6 (d)]{GaSiTo}), the later being universally well-defined since the special fiber of $\mathbb{I}$ (the Orlik--Terao algebra) is Cohen--Macaulay by \cite{ProudSpeyer}).

The section ends with the announced results on the linear type property of the Jacobian ideal $J_f$ that include an affirmative answer to the second of the above conjectures when $n=3$. As a consequence we obtain a characterization of freeness in rank $3$ in terms of the Rees algebra of the corresponding Jacobian ideal.

The third section is largely experimental as it stands.
As it aims at arrangement of forms of degrees $\geq 2$, in whatever analogy exists to the hyperplane case, we introduce conditions on these forms that mimic some of the properties of the latter.
We have chosen to work with a couple of conditions, under the designation of {\em near transversality}.
This includes assuming that every individual form $f_i$ is such that Proj$(R/(f_i))$ is smooth, a sort of extravagant requirement to mimic the fact that a straight line is smooth. It also includes the requirement that the forms be mutually non-tangent, a condition that we chose to encode via a Jacobian matrix in relation to certain primes.
Since the degrees of the forms are quite arbitrary, there is no apparent obstruction to assume only that $m\geq 2$.
The case of two forms is already sufficiently appealing, so we prove a main result under the near transversality assumption (or close to this).
For $m$ arbitrary, we also state a lower bound for the initial degree of the Jacobian ideal of the product of the forms. Finally, under a special condition on the degrees of the forms we are able to obtain the analog of Rose--Terao--Yuzvinsky theorem.

\section{The Jacobian ideal of a hyperplane arrangement}

\subsection{The initial degree of the syzygies}\label{indeg_of_Jacobian}

Let $\mathcal{A}$ denote a central hyperplane arrangement in affine space $\mathbb{K}^n$ over an infinite field $\mathbb{K}$.
Although the notions make sense in positive characteristic, the use of the Euler relation in the main results require that char$(\mathbb{K})$ does not divide the degree of the defining polynomial of the arrangement, a hypothesis that we take for granted throughout refraining from stating it on spot.

For our purpose it will be convenient to think of $\mathcal{A}$ as being given by mutually independent linear forms $l_1,\ldots,l_m$ in the polynomial ring $R:=\mathbb{K}[x_1,\ldots, x_n]$.
We assume throughout that $n\geq 2$ and $m\geq n+1$.

The arrangement is said to be {\em generic} if any $n$ of the defining linear forms are $\mathbb{K}$-linearly independent.
Clearly, an arrangement for which the coefficients of its forms are randomly picked is automatically generic and this conveys a fast way of producing generic arrangements.

Let $f:=l_1\cdots l_m\in R$ be the defining polynomial of the  arrangement.
We will closely focus on the Jacobian (or gradient) ideal $J_f\subset R$ of $f$, i.e., the ideal generated by the partial derivatives $f_{x_i}:=\partial f/\partial x_i \, (1\leq i\leq n)$ of $f$.
By abuse, we let $\mathrm{Syz}(J_f)$ denote the module of first syzygies of $\{f_{x_1},\ldots, f_{x_n}\}$.

Fixing a linear form $l=l_i$ of the arrangement, one has the well-known notions of {\em deletion} $\mathcal A\setminus \{H\}$ and {\em restriction} $\mathcal A^H$ of $\mathcal A$ with respect to a given hyperplane $H=V(l_i)$.
In algebraic terms, deletion is just the arrangement of the forms $l_1,\ldots,l_{i-1},\widehat{l_i}, l_{i+1},\ldots, l_m$ in the same polynomial ring $R$.
Restriction is the arrangement of the residual forms $\{\langle l_i,l_j\rangle/\langle l_i\rangle\,|\; j\neq i\}$ in the residual  ring $R/\langle l_i\rangle$, upon identification of the latter with a polynomial ring over $\mathbb{K}$ in $n-1$ variables.

By this description, it is immediate to verify that if $\mathcal A$ is a generic arrangement, then so are the respective deletion and restriction arrangements with respect to any of the defining linear forms.

Borrowing from the notation of \cite{TaDiSti}, we set $r(\mathcal A):={\rm indeg}(\mathrm{Syz}(J_f))$ for the corresponding initial degree, i.e.,
$$r(\mathcal A):= \min_{r\in\mathbb Z}\{r|(\mathrm{Syz}(J_f))_r\neq 0\}.$$

We will make use of the following theorem:

\begin{thm}\label{Addition-deletion}  {\rm (\cite[Theorem 2.14]{TaDiSti})}  
	Let $l=l_i$ be a linear form defining the hyperplane $H$ of a central arrangement $\mathcal{A}$. If $r(\mathcal A\setminus \{H\}) < r(\mathcal A^H)$, where $H=V(l)$, then
	$r(\mathcal{A})=r(\mathcal A\setminus \{H\}) +1$.
\end{thm}
Though \cite{TaDiSti} assumes ${\rm char}(\K)=0$ throughout, it is rather clear that the arguments only requires that ${\rm char}(\K)$ does not divide $m=\deg(f)$.

It is proved in \cite[Corollary 4.4.3]{RoseTerao} that, in the generic case with $m\geq n+1$, the syzygy module of $J_f$ is equigenerated in degree $m-n+1$.
Of course, this comes out as a conflagration of introducing various chain complexes, none of which is trivial.
As it turns, one can walk a good mile just by knowing that the initial degree of $\syz(J_f)$ is $m-n+1$.
For this reason, and also for the sake of completeness, we state and prove the following:

\begin{prop}\label{exact_indeg_Jacobian_generic}
	Let $\mathcal{A}$ denote a generic central hyperplane arrangement of $m$ linear forms in affine space $\mathbb{K}^n$, satisfying $m\geq n+1$. Then $r(\mathcal{A})=m-n+1.$
\end{prop}
\demo
By the discussions in the Appendix, one can assume that the arrangement is now $\{x_1,\ldots,x_n,l_1,\ldots,l_{m-n}\}\subset R:=\K[x_1,\ldots,x_n]$, for suitable linear forms $l_1,\ldots,l_{m-n}$. In particular, the updated defining polynomial is  $f:=x_1\cdots x_n l_1 \cdots l_{m-n}$.

\medskip

\noindent {\em Starting Case: $m=n+1$.}  Here, say, $f=x_1\cdots x_n l,$ with $l=\alpha_1 x_1 + \cdots + \alpha_n x_n$, where $\alpha_1,\ldots,\alpha_n\in\mathbb K\setminus\{0\}$ since the arrangement is generic.
As easily seen, one has
\begin{equation}\label{partial_initial_case}
f_{x_i}=x_1\cdots \widehat{x_i}\cdots x_n (\alpha_ix_i+l).
\end{equation}
Therefore, two indices $i\neq j$ yield an obvious reduced Koszul relation of degree $2$.
Thus, it suffices to show that the partial derivatives have no syzygy of degree $\leq 1$.

\smallskip

{\sc Claim 1.} $\{f_{x_1},\ldots, f_{x_n}\}$ is linearly independent over $\mathbb{K}$.

This is a far more general result, but in the present case it follows simply because if one has scalars $\lambda_i\in\mathbb{K}$, not all zero, such that
 $$\lambda_1(x_2\cdots x_nl+x_1\cdots x_n \alpha_1)+\cdots +\lambda_n(x_1\cdots x_{n-1}l +x_1\cdots x_n \alpha_n)=0$$
 then a scalar $\lambda\in \mathbb{K}$ comes out such that
 $$\lambda\, x_1\cdots x_n= \left(\sum_i \lambda_i \,x_1\cdots \widehat{x_i}\cdots x_n\right)\, l.$$
Now, either $\lambda=0$, in which case the parenthetical content would have to vanish; since some $\lambda_i\neq 0$ by assumption, this would say that the partial derivatives of $x_1\cdots x_n$ are linearly dependent, an absurd since, as is well-known, they are even algebraically independent over $\mathbb{K}$.
We are left with the option that $\lambda\neq 0$, and hence $x_i$ would have to be a factor of $l$, because $\lambda_i\neq 0$ so $x_i$ cannot divide the parenthetical summation.

So much for this claim.

{\smallskip

{\sc Claim 2.} $\{f_{x_1},\ldots, f_{x_n}\}$ admits no linear syzygy.

Suppose that $L_1,\ldots,L_n \in R$ are linear forms, not all zero, such that $$L_1f_{x_1}+\cdots +L_nf_{x_n}=0.$$

	Then,
	\begin{equation}\label{basic_relation}
	(L_1x_2\cdots x_n + \cdots +L_nx_1\cdots x_{n-1})l=-x_1\cdots x_{n}(\alpha_1 L_1+\cdots +\alpha_nL_n).
	\end{equation}
	
There are two cases to have in mind:

{\em First case}: $\alpha_1 L_1+\cdots +\alpha_nL_n=0$.

This implies that $(L_1,\ldots,L_n)^t$ is a linear syzygy of the Boole arrangement of equation $x_1\cdots x_n$. But the syzygies of the latter are well-known to be given by the Hilbert--Burch matrix
$$\left[\begin{matrix}
x_1 & 0 & 0 & \cdots\\
-x_2 & x_2 & 0 & \cdots\\
0  & -x_3 & x_3 & \cdots\\
0 & 0 & -x_4 & \cdots\\
\vdots & \vdots & \vdots &\cdots
\end{matrix}
\right]
$$
Writing up this relation yields
$$L_1=\beta_1x_1, L_2=-(\beta_1+\beta_2)x_2,\ldots, L_i=\pm (\beta_{i-1}+\beta_i)x_i, \ldots, L_n=-\beta_{n-1}x_n,$$
which, when plugged back into $\alpha_1 L_1+\cdots +\alpha_nL_n=0$, implies a relation
$$\alpha_1\beta_1x_1 -\alpha_2(\beta_1+\beta_2)x_2+\cdots + \alpha_{n-1}(\beta_{n-1}+\beta_{n-1})x_{n-1}-\alpha_n\beta_{n-1}x_n=0.$$
Since some $L_i\neq 0$, the corresponding $\beta$-like coefficient is nonzero. But $\alpha_i\neq 0$ for any $i$. This would give a $\mathbb{K}$-linear dependence relation for the variables.

\smallskip

{\em Second case} : $L:=\alpha_1 L_1+\cdots +\alpha_nL_n\neq 0$.

In this case, $L$ must divide one of the factors on the left side of the relation (\ref{basic_relation}).
If $L$ divides the parenthetical summation, after canceling, one gets an equality $P\,l=x_1\cdots x_n$, for some polynomial $P\in R$.
Since no $x_i$ divides $l$ by the generic shape of $l$, this is an absurd because of the UFD property.
Therefore, we must have, say, $L=\beta l$, for some nonzero $\beta\in \mathbb{K}$.
Moreover, upon substitution in (\ref{basic_relation} and cancellation, one obtains
$$L_1x_2\cdots x_n + \cdots +L_nx_1\cdots x_{n-1}=-\beta x_1\cdots x_{n}.$$
This means that $L_1,\ldots,L_n$ are the coordinates of a derivation of $R/\langle x_1\cdots x_n\rangle$, hence
 $L_i=\beta_i x_i$, for certain $\beta_i\in \mathbb{K}$. Canceling $x_1\cdots x_n$ yields $\beta_1+\cdots +\beta_n=-\beta$, which in turns yields the relation $\alpha_1\beta_1x_1+\cdots + \alpha_n\beta_nx_n=-L=-\beta l$ and hence, a syzygy
$\alpha_1(\beta+\beta_1)x_1+\cdots \alpha_n(\beta+\beta_n)x_n=0$ with scalar coefficients.
Since no $\alpha_i$ vanishes, every term $\beta+\beta_i$ must. Together with the relation $\beta_1+\cdots +\beta_n=-\beta$ this easily leads to $\beta_i=0$ for all $i$, hence  also $L_i=0$ for all $i$, as was to be shown.

\medskip

\noindent {\em Inductive Step}: $m\geq n+2$.

We will induct on $m$.
Recall that both the deletion and the restriction are still generic arrangements, and $m$ drops on either. Take both with respect to $l=x_1$.
Thus, $r(\mathcal{A}\setminus \{H\})=(m-1)-n+1=m-n$ by the inductive hypothesis.
Likewise,  the defining polynomial associated to the restriction $\mathcal{A}^H$ is $\widetilde{f}=x_2\cdots x_n \widetilde{l_1} \cdots \widetilde{l_{m-n}}$, with $\widetilde{l_i}=l_i|_{x_1=0}$, and one has  $r(\mathcal{A}^H)=(m-1)-(n-1)+1=m-n+1$ by the inductive hypothesis.
Consequently,
$r(\mathcal{A}\setminus \{H\})<r(\mathcal{A}^H).$
Then, by
 Theorem~\ref{Addition-deletion} we have $r(\mathcal{A})=r(\mathcal{A}\setminus \{H\})+1=m-n+1.$
\qed

\subsection{The Rose--Terao--Yuzvinsky theorem under an algebraic angle}\label{Yuz}

We keep the notation of the previous section, under the assumption that $m$ does not divide ${\rm char}(\K)$, a proviso we will refrain from repeating throughout.

Given the defining polynomial $f=l_1\cdots l_m\in R$ of a generic central arrangement $\mathcal{A}$, under a slight danger of confusion, we set $f_i:=f/l_i$ (i.e., $f_i$ is not the partial derivative $f_{x_i}$ as might sometimes be the chosen notation).
Consider the ideal
$$\mathbb I=\mathbb{I}(\mathcal{A}):=\langle f_1,\ldots,f_m\rangle=\langle l_2\cdots l_m,\ldots,l_1\cdots l_{m-1}\rangle\subset R,$$
generated by these $(m-1)$-fold products of the linear forms $l_1,\ldots,l_m$.
Note that this is a minimal set of generators of $\mathbb I$ (see, e.g., \cite[Lemma 3.1 (a)]{GaSiTo}).

\begin{thm}\label{saturation_and_regularity}
	Let $f=l_1l_2\cdots l_m$ as above denote the defining polynomial of a generic central arrangement in $R=\mathbb{K}[x_1,\ldots,x_n]$ with $m\geq n+1$.
	 Then:
	\begin{enumerate}
		\item[{\rm (i)}]  $[\mathbb I]_{2m-n-1}=[J_f]_{2m-n-1}$.
		\item[{\rm (ii)}] If $n\geq 3$ then$(J_f)^{\rm sat}=\mathbb I$. In particular, $J_f$ and $\mathbb{I}$ coincide locally everywhere on the punctured spectrum $\spec R\setminus \{\fm\}$.
		\item[{\rm (iii)}] ${\rm reg}(J_f)={\rm sat}(J_f)=2m-n-1$, where ${\rm sat}(I)$ denotes the saturation exponent of an ideal $I\subset R$..
		\end{enumerate}
\end{thm}
\begin{proof} Let $\mathbb I:=\langle l_2\cdots l_m,\ldots,l_1\cdots l_{m-1}\rangle$ be the ideal generated by all $(m-1)$-fold products of the linear forms $l_1,\ldots,l_m$. Clearly, $\mathbb I$ has height $2$ and, as one easily verifies, $J_f\subset \mathbb I$.
	Therefore, $J_f$ has height at most $2$ and, since $f$ is a reduced form,  the height is exactly $2$ (actually, it can be seen that $J_f$ and $\mathbb I$ have the same minimal primes, all of codimension $2$).

(i)
One inclusion being obvious, it suffices to show that $[\mathbb I]_{2m-n-1}\subset [J_f]_{2m-n-1}$.
Now, let $P\in \mathbb I$, with $\deg(P)=2m-n-1$. Write
\begin{equation}\label{P-writing}
P=(l_2\cdots l_m)P_1+\cdots +(l_1\cdots l_{m-1})P_m,
\end{equation}
for certain forms $P_i\in [\fm^{m-n}]_{m-n}$,  where $\frak m=(x_1,\ldots,x_n)$.

\smallskip

{\sc Claim 1.} ($\mathcal{A}$ generic) \,Let $\mathbb{I}_{m-n}(\mathcal{A})$ denote the ideal of $R$ generated by all $(m-n)$-fold products of the given linear forms $l_1,\ldots,l_m$. Then $\mathbb{I}_{m-n}(\mathcal{A})=\frak m^{m-n}$.

The proof is by induction on $m$. It is quite obvious for $m=n+1$. Thus, assume that $m\geq n+2$ and, for every $1\leq i\leq m$, consider the deletion $\mathcal{A}_i:=\mathcal{A}\setminus l_i$.
By the inductive hypothesis, $\mathbb{I}_{m-1-n}(\mathcal{A}_i)=\fm^{m-1-n}$, hence
$$l_i\,\fm^{m-1-n}= l_i\, \mathbb{I}_{m-1-n}(\mathcal{A}_i)\subset \mathbb{I}_{m-n}(\mathcal{A}),$$
for all $i$. But since $\{l_1,\ldots,l_n\}$ generates $\fm$ then $\fm^{m-n}= \langle l_1,\ldots,l_n\rangle \fm^{m-1-n}$, so we are through..

(A more conceptual argument follows by noting that, since the linear forms are generic, the linear code dual to $l_1,\ldots,l_m$ has minimum distance $m-n+1$, hence the assertion follows  from \cite[Theorem 3.1]{To1}).

Thus, for each $i=1,\ldots,m$, write
$$P_i=\sum_{1\leq j_1<\cdots<j_{m-n}\leq m} c^i_{j_1,\ldots,j_{m-n}}l_{j_1}\cdots l_{j_{m-n}},$$
for certain coefficients $c^i_{j_1,\ldots,j_{m-n}}\in\K$.

Now, in the above summation for fixed $i$, if a term effectively involves the form $l_i$ then its product with the corresponding product $l_1\cdots \hat{l_i}\cdots l_m$ up in (\ref{P-writing}) is a multiple of $f$, hence belongs to $J_f$ by Euler's relation.
We have to deal with case where a term does not effectively involve $l_i$.
Because of the obvious symmetry of the problem, there is no loss of generality in taking $i=m$ and assuming that $P_m$ has a term not multiple of $l_m$.

We then show:

{\sc Claim 2.} The $(2m-n-1)$-form $\Delta:=(l_1\cdots l_{m-n})^2l_{m-n+1}\cdots l_{m-1}$ belongs to $J_f$.

The proof is again by induction on $m\geq n+1$.

If $m=n+1$,  $\Delta=l_1^2l_2\cdots l_n\in J_f$.
As in the proof of Proposition~\ref{exact_indeg_Jacobian_generic}, up to a change of coordinates we can assume that $l_i=x_i$ for $ i=1,\ldots,n$ and  $l_{n+1}=\ell=\alpha_1x_1+\cdots+\alpha_nx_n$, with $\alpha_i\in\mathbb K\setminus\{0\}$.
Thus, $f=x_1\cdots x_n\ell$ and
 $f_{x_1}=x_2\cdots x_n\ell+x_1x_2\cdots x_n\alpha_1$ by (\ref{partial_initial_case}).
Clearly then, $$x_1^2x_2\cdots x_n=\frac{1}{\alpha_1}(x_1f_{x_1}-f)\in J_f.$$

 Let now  $m\geq n+2$. Write
$$\Delta=(l_1)^2\underbrace{(l_2\cdots l_{m-n})^2l_{m-n+1}\cdots l_{m-1}}_{\Delta'}.$$
By the inductive hypothesis, $\Delta'\in J_{f'}$, where $f'=l_2\cdots l_m$.
Write $\Delta'=Q_1 f'_{x_1}+\cdots+ Q_n f'_{x_n}$, for certain $Q_i\in R$. But, for $i=1,\ldots,n$, one has $f_{x_i}=(l_1)_{x_i}f'+l_1 f'_{x_i},$ and therefore $(l_1)^2 f'_{x_i}\in J_f$ for $i=1,\ldots,n$. But then, $\Delta=l_1^2\Delta'\in J_f,$ as was to be shown.

\medskip

(ii) This follows from (i) since any form in $[\fm^{m-n}]_{m-n}$ multiplies any $f/l_i$ into $[\mathbb I]_{2m-n-1}$, thus giving
 $\mathbb I\subset (J_f)^{\rm sat}$. The reverse inclusion is obvious since $\mathbb I$ is saturated, being a codimension two perfect ideal and $n\geq 3$.

 \medskip

(iii) As mentioned, ${\rm sat}(I)$ denotes  the {\em saturation exponent} (or the {\rm satiety}) of an ideal $I\subset R$.

{\sc Claim.} ${\rm reg}(J_f)={\rm sat}(J_f).$

It is well known that
$${\rm reg}(J_f)=\max\{{\rm sat}(J_f), {\rm reg}(J_f^{\rm sat})\}.$$
 By Proposition~\ref{exact_indeg_Jacobian_generic}, at least ${\rm reg}(J_f)\geq 2m-n-1>m-1$, since $m>n$.
On the other hand, assume that $n\geq 3$. Then  ${\rm reg}(J_f^{\rm sat})={\rm reg}(\mathbb I)$ by (ii).
But, since $\mathbb{I}$ is a linearly presented codimension two perfect ideal, ${\rm reg}(\mathbb I)=m-1$.
This implies the claim if $n\geq 3$.
If $n=2$ then $J_f$ is $\fm$-primary, hence $J_f^{\rm sat}=(1)$, whose regularity is zero by convention. Thus, we are home in this case too.

\smallskip

Suppose ${\rm sat}(J_f)\geq 2m-n$. Then, by \cite[Proposition 2.1]{BeGi}, there exists a minimal generator $h\in J_f:\fm\setminus J_f$, of degree $d:={\rm sat}(J_f)-1$. Thus, $d\geq 2m-n-1$.

But $(J_f:\fm)\subset (J_f)^{\rm sat}=\mathbb I$, so $h=(l_2\cdots l_m)P_1+...+(l_1\cdots l_{m-1})P_m$, for some $P_j\in R$ of degree $d-m+1\geq m-n$. Thus, each $P_j$ is a polynomial combination of monomials of degree $m-n$. An analogous argument to the proof of  Claim 2 in the item (i) yields $h\in J_f$, contradicting our choice of $h$.
Therefore, ${\rm sat}(J_f)\leq 2m-n-1$, hence ${\rm reg}(J_f)=2m-n-1.$
\end{proof}

\medskip

Let $f:=l_1\cdots l_m\in R:=\mathbb K[x_1,\ldots,x_n]$ be the defining polynomial of a central hyperplane arrangement $\mathcal{A}$ in $\mathbb{K}^n$. The module of {\em logarithmic derivations} associated to $f$ is defined as $\mathrm{Derlog}(\mathcal{A}):=\{\theta\in {\rm Der}(R)| \theta (f)\subset \langle f \rangle\}.$
If ${\rm char}(\mathbb K)$ does not divide $m=\mathrm{deg}(f)$, upon the identification ${\rm Der}(R)=R^n$ this module has  a splitting structure
\begin{equation}\label{splitting}
\mathrm{Derlog}(\mathcal{A})= \mathrm{Syz}(J_f)\oplus R\theta_E,
\end{equation}
where $J_f\subset R$ denotes the Jacobian ideal of $f$ and $\theta_E$ is the Euler derivation.
Here, as discussed earlier, $\mathrm{Syz}(J_f)\subset R^n$ denotes the module of first syzygies of the partial derivatives of $f$.

As a consequence of Theorem~\ref{saturation_and_regularity}, we now retrieve part of the findings of Rose--Terao and Yuzvinsky.

\begin{thm} \label{Yuz}
	Let $f=l_1l_2\cdots l_m$ as above denote the defining polynomial of a generic central arrangement in $R=\mathbb{K}[x_1,\ldots,x_n]$ with $m\geq n+1$.
 Then
\begin{enumerate}
	\item[{\rm (i)}] ${\rm depth}(R/J_f)=0$.
	\item[{\rm (ii)}]	The minimal graded free resolution of $J_f$ has the following shifts
	{\small
	$$0\rightarrow R^{b_n}(-(2m-1))\rightarrow R^{b_{n-1}}(-(2m-2))\rightarrow\cdots\rightarrow R^{b_1}(-(2m-n))\rightarrow R^n(-(m-1)).$$
}
\end{enumerate}
\end{thm}
\begin{proof}
(i)
The assertion is trivial for $n=2$ since $J_f$ is $\fm$-primary.
Thus, assume that $n\geq 3$ and recall that $J_f$ has height two.

Take a reduced primary decomposition $J = J^{\rm un} \cap \bigcap_i \mathcal{N}_i,$ where $J^{\rm un}$ is the unmixed part of $J$ and each $\mathcal{N}_i$ is a primary component of codimension at least $3$.
	
If we assume that $\depth (R/J) > 0$, then $3\leq \codim \mathcal{N}_i<n$, for every $i$.
Therefore, by Theorem~\ref{saturation_and_regularity} (ii), $\mathbb{I} = J : \fm^{\infty} =  (J^{\rm un} : \fm^{\infty} )\cap \bigcap_i (\mathcal{N}_i : \fm^{\infty}) = J^{\rm un} \cap \bigcap_i \mathcal{N}_i$.
Therefore,  $J=\mathbb{I}$, which is nonsense since the minimal number of generators of $\mathbb I$ is at leat $n+1$.

\smallskip
(ii)
By definition, ${\rm reg}(J_f)$ is the maximum value $\beta_k-k$, where $\beta_k$ is the maximum shift at step $k$ in the minimal graded free resolution of $J_f$.
By Proposition~\ref{exact_indeg_Jacobian_generic}, the minimum shift at step $k=1$ is $r(\mathcal A)+(m-1)=2m-n$, which equals ${\rm reg}(J_f)+1$ by Theorem~\ref{saturation_and_regularity} (iii). But the maximum shift at each step must increase by at least one and the homological dimension of $R/J_f$ over $R$ is $n$ by item (i). This implies the claimed result by a standard argument (see, e.g., \cite[Corollary 9 (2)]{Ooishi}).
\end{proof}

\begin{rem}
	(1) Proposition~\ref{exact_indeg_Jacobian_generic}, Theorem~\ref{saturation_and_regularity} and Theorem~\ref{Yuz} depend strongly on the hypotheses that $m\geq n+1$ and that the arrangement be generic.
	The failure is pretty bad if $m\leq n$, while genericity cannot be replaced by  assuming, e.g.,  that all subsets of $n$ linear forms, with one exception, are linearly independent. For that, consider the following simple example in $\K[x,y,z]$: $\mathcal{A}:\{x,y,z, x+y\}$,  a free arrangement (i.e., $J_f$ is a perfect ideal); here, the initial degree of $\syz(J_f)$ is $1<2=m-n+1$.
	
(2) In \cite{RoseTerao} the ground field $\K$ is assumed to be of characteristic zero throughout, but the authors remark at the end (Section 4.6) that the result is valid provided ${\rm char}(\K)$ does not divide  $m=\deg(f)$.
As to our approach, the vast majority of the above results crumbles down without this proviso.
A simplest example is as follows: let $R=\K[x,y,z]$, where ${\rm char}(\K)=2$ and let
$$\mathcal{A}=\{x,y,z,x+y+z\}.$$
Clearly, $\mathcal{A}$ is generic.

Calculation with \cite{M2} yields:
\begin{itemize}
		\item $f \not\subset J_f$.
	\item $J_f$ is perfect; in particular, $J_f^{\rm sat}=J_f\neq \mathbb{I}$.
	\item The free resolution of $J_f$ has the shape $0 \rar R(-4)\oplus R(-5) \rar R(-3)^3 \rar J_f\rar 0$; in particular, $J_f$ has a syzygy of degree $1< 2 = m-n+1$ -- in fact, $2$ is the degree of a minimal syzygy of maximal degree, which happens to be given by a reduced Koszul syzygy.
	\item Derlog$(\mathcal{A})$ does not split as in (\ref{splitting}); in fact, it admits many generators of degree $2$.
	\item Derlog$(\mathcal{A})$ has homological dimension $1$, while Syz$(J_f)$ is a free module.
    \item $J_f$ is not a reduction of $\mathbb{I}$ since $J_f\colon \mathbb{I}^{\infty}\neq (1)$.
\end{itemize}
We note that characteristic $2$ is not the issue, as similar examples are easily handled in any positive characteristic.
\end{rem}

\section{The Jacobian ideal as a minimal reduction}

\subsection{Approach via the Orlik--Terao algebra} Consider the  Orlik-Terao algebra of $\mathcal{A}$, which as known is graded-isomorphic to the special fiber $\mathcal{F}(\mathbb{I})$ of $\mathbb{I}$ (see \cite{GaSiTo}). Let $\mathcal{F}(\mathbb{I}) \simeq \K[T_1, \ldots, T_m]/Q$ denote a presentation.

\subsubsection{Rank plus one}

The following lemma is an adaptation to the ideal $J_f$ of a well-known fiber criterion for reductions (see, e.g., \cite{SwanHu}).
Let $a_{i,j}: =(l_j)_{x_i}:= \partial l_j /\partial x_i$ be the $x_i$-coefficient of $l_j$, and let $T_1,\ldots,T_m$ be presentation variables as above.

\begin{lem}\label{reductionOT} The following are equivalent:
	\begin{enumerate}
		\item $J_f$ is a minimal reduction of $\mathbb{I}$.
		\item The set $\{\sum_{j=1}^m a_{1,j}T_j,\ldots,\sum_{j=1}^m a_{n,j}T_j \}$ is a regular sequence on $\mathcal{F}(\mathbb{I})$.
		\item The ideal $\langle\sum_{j=1}^m a_{1,j}T_j,\ldots,\sum_{j=1}^m a_{n,j}T_j,Q\rangle\subset \K[T_1,\ldots,T_m]$ is $\langle T_1,\ldots, T_m\rangle$-primary.
	\end{enumerate}
\end{lem}

With this at hand, one can easily handle the case where $m=n+1$.

\begin{lem} \label{nplus1} Let $\mathcal{A}=\{l_1,\ldots,l_{n+1}\}$ be a central arrangement of rank $n$. Let $f=l_1\cdots l_{n+1}$ and $\mathbb{I}$ be as before. Then $J_f$ is a minimal reduction of $\mathbb{I}$.
\end{lem}
\begin{proof}
	As done once and over, since $\{l_1,\ldots,l_m\}$ span $[R]_1$, up to a change of variables we can assume that $l_i=x_i$ for $1\leq i\leq n$ and $l_{n+1}=\alpha_1x_1+\cdots+\alpha_nx_n$,
	
	Thus,
	$f=x_1\cdots x_n (\alpha_1x_1 +\cdots+\alpha_nx_n)$. Letting $f_i=f/l_i,$ for $i=1,\ldots,n+1$, one has $f_{x_i}=f_i+\alpha_if_{n+1},$ for $i=1,\ldots,n+1.$ Note that $l_{n+1}=\alpha_1l_1+\cdots+\alpha_nl_n$ is the dependency that gives the only minimal generator of the Orlik-Terao ideal. Suppose $\alpha_{j+1}=\cdots=\alpha_{n}=0$ for some $j\geq 2$ and $\alpha_1,\ldots,\alpha_j\neq 0.$ That generator is $Q=T_1\cdots T_j - (\alpha_1T_2\cdots T_jT_{n+1}+\cdots +\alpha_jT_2\cdots T_{j-1}T_{n+1}).$ So, the ideal in the third part of Lemma~\ref{reductionOT} is
	{\footnotesize
		\begin{eqnarray}
		\langle T_1+\alpha_1T_{n+1},\cdots,T_j+\alpha_jT_{n+1},T_{j+1},\cdots,T_n,T_1\cdots T_j- (\alpha_1T_2\cdots T_jT_{n+1}+\cdots +\alpha_jT_2\cdots T_{j-1}T_{n+1})\rangle\nonumber\\
		=\langle T_1+\alpha_1T_{n+1},\cdots,T_j+\alpha_jT_{n+1},T_{j+1},\cdots,T_n,(1-(-1)^{j-1}j)\alpha_1\cdots\alpha_{j}T_{n+1}^j\rangle\nonumber
		\end{eqnarray}
	}
	which is indeed $\langle T_1\,\ldots,T_n\rangle$- primary. By Lemma~\ref{reductionOT},  $J_f$ is a minimal reduction of $\mathbb{I}$.
\end{proof}

\subsubsection{Rank $2$}

\begin{prop}\label{rank_two}
Let $f\in R=\K[x,y]$ be the defining polynomial of an arrangement of rank $2$ and size $m\geq 3$. Then $J_f$ is a reduction of $\mathbb I$ with reduction number one.
\end{prop}
\demo It suffices to prove that $\mu(\mathbb{I}^2)=\mu(J_f\mathbb{I})$ since $J_f\mathbb{I}\subset \mathbb{I}^2$ and both ideals are equigenerated in the same degree.
Now, since $J_f$ is generated by a regular sequence (of two forms), then $\mu(J_f\mathbb{I})=2\mu(\mathbb I)-{2 \choose 2}=2m-1$ (see, e.g., \cite[Proposition 2.22]{fiber}).

On the other hand, one has $\mu(\mathbb{I}^2)={{m+1} \choose 2}-\mu(Q)$, where $Q$ is as above the defining ideal of the Orlik--Terao algebra.
It is known that $Q$ is generated by a minimal set of generators among the circuits corresponding to the dependencies -- all of size $3$ in the rank $2$ case, hence the generators are quadrics.
This minimal set has ${m-1 \choose 2}$ elements.
Therefore, one gets
$$\mu(\mathbb{I}^2)={{m+1} \choose 2}- {m-1 \choose 2}=2m-1,$$
as was to be shown.
\qed

\subsection{Partition method}

Let $R=\K[x_1,\ldots,x_r; y_1,\ldots,y_s]$ denote a polynomial ring in $r+s$ variables.
Consider  a central arrangement $\mathcal{A}$ of $\K^{r+s}=\spec{R}$ of size $m$.
Assume that $\mathcal{A}=\mathcal{B}\cup \mathcal{C}$ is a $2$-partition,
where $\mathcal{B}$ (respectively, $\mathcal{C}$) is an arrangement of $\K^r=\spec{\K[x_1,\ldots,x_r]}$ of size $m_x$ (respectively, is an arrangement of $\K^s=\spec{\K[y_1,\ldots,y_s]}$ of size $m_y$), so that $m=m_x+m_y$.

One assumes that $m_x\geq r, m_y\geq s$.

Let $f$ and $g$ denote, respectively, the polynomials of $\mathcal{B}$ and $\mathcal{C}$.
Thus, $F:=fg$ is the polynomial of the total arrangement $\mathcal{A}$.
Let $\mathbb{I}_f\subset \K[x_1,\ldots,x_r]$ denote the ideal of $(m_x-1)$-fold products of the first arrangement; similarly, define $\mathbb{I}_g\subset \K[y_1,\ldots,y_s]$ and $\mathbb{I}_F\subset R$.

Finally, let $J_f\subset \K[x_1,\ldots,x_r]$ denote the Jacobian ideal of $f$; similarly, take $J_g\subset \K[y_1,\ldots,y_s]$ and $J_F\subset R$.

\begin{lem}\label{partition_basics}
	One has $\mathbb{I}_F=\langle f\mathbb{I}_g,g\mathbb{I}_f\rangle$ and $J_F=\langle fJ_g,gJ_f\rangle$ as ideals in $R$.
\end{lem}
\demo This is immediate from the data.
\qed

\begin{prop}
	Suppose that $J_f$ (respectively, $J_g$) is a reduction of $\mathbb{I}_f$ of reduction number $\leq r-1$ (respectively, of $\mathbb{I}_g$ of reduction number $\leq s-1$).
	Then $J_F$ is a reduction of $\mathbb{I}_F$ of reduction number $\leq r+s-1$.
\end{prop}
\demo
The assumption is that $\mathbb{I}_f^r=J_f\mathbb{I}_f^{r-1}$ and $\mathbb{I}_g^s=J_g\mathbb{I}_g^{s-1}$ and we wish to prove that $\mathbb{I}_F^{r+s-1}=J_F\mathbb{I}_F^{r+s-2}$.
Quite generally, $J_F\subset \mathbb{I}_F$, hence it suffices to show the inclusion $\mathbb{I}_F^{r+s-1}\subset J_F\mathbb{I}_F^{r+s-2}$.

By Lemma~\ref{partition_basics} and the `binomial expansion',  $\mathbb{I}_F^{r+s-1}$ is spanned by the ideals of the shape $\mathfrak{a}_t:=f^{r+s-1-t}\mathbb{I}_g^{r+s-1-t}g^t\mathbb{I}_f^t$, for $0\leq t\leq r+s-1$.

As usual, if $t<r$ then $r+s-1-t\geq s$.
Thus, it will suffice to assume that $r+s-1-t\geq s$.
In this case, one has $\mathbb{I}_g^{r+s-1-t}=J_g \mathbb{I}_g^{r+s-2-t}$ by the assumption, hence
$$\mathfrak{a}_t=f^{r+s-1-t}\mathbb{I}_g^{r+s-1-t}g^t\mathbb{I}_f^t=(f J_g) (f\mathbb{I}_g)^{r+s-2-t} (g^t\mathbb{I}_f^t)\subset J_F\mathbb{I}_F^{r+s-2-t} \mathbb{I}_F^t=J_F \mathbb{I}_F^{r+s-2},$$
as was to be shown.
\qed

\smallskip

In order to derive the consequence below one emphasizes that the above proposition applies to any $2$-partition of $\mathcal{A}$, not necessarily one given by components.

\begin{cor}\label{partition_general}
	In the above notation, if $J_f$ is a reduction of $\mathbb{I}_f\subset R$ for any non-decomposable central arrangement of size $m\geq \dim R$, then the same holds for any central arrangement whatsoever and the reduction number is at most the size of the arrangement less the number of components..
\end{cor}
\demo
Given a full decomposition $\mathcal{A}=\mathcal{A}_1\cup \mathcal{A}_2\cdots\cup \mathcal{A}_t$ of a central arrangement into indecomposable arrangements, apply the proposition to $\mathcal{A}_1\cup \mathcal{A}_2$, then to $(\mathcal{A}_1\cup \mathcal{A}_2)\cup \mathcal{A}_3$, and so on so forth.
\qed

Given a central arrangement $\mathcal{A}=\{l_1,\ldots,l_m\}$ of rank $n$, a form $l_i$ is a {\em coloop} if the corresponding deletion has rank $n-1$.

\begin{cor}\label{coloop} Let $R=\K[x,y,z]$. If $\mathcal A$ admits a coloop then $J_f$ is a minimal reduction of $\mathbb I$.
\end{cor}
\demo By a change of coordinates, we may assume that the coloop is $z$.
Therefore, one has a partition of $\mathcal A$ following the separation $R=\K[x,y; z]$.
The case of rank one is trivial (or vacuous), while the case of rank two follows from Proposition~\ref{rank_two}.
\qed

\subsection{Main theorem}

Once again, let $\mathbb I\subset R:=\mathbb K[x_1,\ldots,x_n]$ stand for the ideal generated by all $(m-1)$-fold products of the linear forms $l_1,\ldots,l_m, m\geq n+1$. We assume throughout that $n\geq 2$ as before.
Setting $f=l_1\cdots l_m$, we prove that the Jacobian ideal $J_f$ is a minimal reduction of $\mathbb{I}$, when $\mathcal A$ satisfies some genericity conditions.

Recall that the arrangement defined by these linear forms is called {\em generic} if any $n$ of them is $\K$-linearly independent. This is equivalent to requiring that the Jacobian matrix of these forms has no vanishing $n$-minor. In particular, the set $\{l_1,\ldots,l_n\}$ has only one component, where the number of components of the arrangement is defined with respect to the variables $x_1,\ldots,x_n$. By \cite[Proposition 2.7 (d)]{GaSiTo}, the reduction number of $\mathbb I$ is $n-u$, where $u$ is the number of such components.

Another important number measures the amount of genericity of the arrangement.
In the arbitrary case, one talks about the arrangement being  {\em $d$-generic} if any $d$ among the defining linear forms are $\K$-linearly independent.

Our standing assumption in this part is that the arrangement $\mathcal{A}$ is $(n-1)$-generic. This includes the arbitrary rank $3$ case since any central arrangement is $2$-generic by definition.

Recall the notation

\begin{notation0}\label{notation_main}
$R=\K[x_1,\ldots,x_n]$.
	
$\bullet$ $\mathbb I\subset R$ is generated by $L_1,\ldots,L_m$, where $L_i:=l_{1}\cdots \widehat{l_i}\cdots l_m,~i=1,\ldots,m$.

$\bullet$ With $f=l_1\cdots l_m$, For each $j=1,\ldots,n$, one has
	$$f_{x_j}=a_{1,j}l_2\cdots l_m +  \cdots +a_{i,j}l_1\cdots \widehat{l_i}\cdots l_m+ \cdots + a_{m,j}l_1\cdots l_{m-1},$$ where $a_{i,j}:=(l_i)_{x_j}$, for $i=1,\ldots,m$.
		
$\bullet$ $\fm=\langle x_1,\ldots,x_n\rangle$.
\end{notation0}

\subsubsection{A key lemma}

In the above notation, one has:

\begin{lem}\label{bottom} {\rm (${\rm char}(\K)\nmid m$)}\, If $\mathcal{A}$ is $(n-1)$-generic then
$L_1L_2\cdots L_n\in J_f\mathbb I^{n-1}$.
\end{lem}
\begin{proof} Set $P:=L_1L_2\cdots L_n=(l_{1}l_{2}\cdots l_{n})^{n-1}(l_{n+1}\cdots l_m)^n$ and consider the two cases as to when $\{l_{1},l_{2},\ldots, l_{n}\}$ has rank either $n$ or $n-1$.
	
	\medskip
	
\noindent	{\sc Rank $n$.}
Note that $P\subset \langle f\rangle$ and set $P=fQ$.
	We contend that $\fm Q\subset \mathbb{I}^{n-1}$.
	
	For this, since $\fm=\langle l_1,\ldots,l_n \rangle$ in this case, it suffices to show that for each $j \in \{1,\ldots,n\}$, one has $l_{j}Q\subset \mathbb{I}^{n-1}$.
	But, indeed one easily sees that $l_{j}Q=
		L_{1}\cdots \widehat{L_{j}}\cdots L_{n}\in \mathbb{I}^{n-1}$.
		
	To conclude, apply the Euler relation of $f$, thus getting
	$mP=mfQ\in \fm J_f \mathbb{I}^{n-1}$.
	
	\medskip
	
\noindent	{\sc Rank $n-1$.}
 Up to reordering, $l_n$ is a $\K$-linear combination of $\{l_1,\ldots,l_{n-1}\}$. Then, up to a change of variables we can, and will, suppose that $l_i=x_i$ for $1\leq i\leq n-1$, so
	$$P= (x_1\cdots x_{n-1}l_{n})^{n-1}(l_{n+1}\cdots l_p l_{p+1}\cdots l_m)^n,$$
where $p\geq n$ is such that $l_n,\ldots,l_p\in\mathbb K[x_1,\ldots,x_{n-1}]$ and ${\rm rank}(l_1,\ldots,l_{n-1},l_k)=n$ for $p+1\leq k\leq m$.

Now, for $j=1,\ldots,n-1$, set
	\begin{eqnarray*}
\Delta_j&:=& x_1\cdots \widehat{x_j}\cdots x_{n-1}l_nl_{n+1}\cdots l_m\\
&=&f_{x_j}-x_{1} \cdots x_{j} \cdots x_{n-1}(a_{n,j}l_{n+1}\cdots l_m)\\
&-&x_{1} \cdots x_{j} \cdots x_{n-1}\left(\sum_{s=n+1}^{m}  a_{s,j}l_{n}l_{n+1}\cdots \widehat{l_s}\cdots l_m)\right)
	\end{eqnarray*}
and introduce it  in the expression of $P$ to enhance its dependence on $x_j$:		
	\begin{eqnarray*}
	P&=&x_j^{n-1}(x_{1}\cdots\widehat{x_{j}}\cdots x_{n-1}l_{n})^{n-2}(l_{n+1}\cdots l_m)^{n-1}\Delta_j\\
	&=&x_j^{n-1}(x_{1}\cdots\widehat{x_{j}}\cdots x_{n-1}l_{n})^{n-2}(l_{n+1}\cdots l_m)^{n-1}f_{x_j}\\
	&&-(x_{1} \cdots x_{j} \cdots x_{n-1})^{n-1}x_{j}(a_{n,j}l_n^{n-2}(l_{n+1}\cdots l_m)^n)\\
	&&-(x_{1} \cdots x_{j} \cdots x_{n-1}l_n)^{n-1}x_{j}\left(\sum_{s=n+1}^{p}  a_{s,j}l_s^{n-1}(l_{n+1}\cdots \widehat{l_s}\cdots l_m)^{n}\right)\\
	&&-(x_{1} \cdots \widehat{x_{j}} \cdots x_{n-1})^{n-1}x_{j}^{n}\left(\sum_{s=p+1}^{m}  a_{s,j}l_s^{n-1}(l_{n+1}\cdots \widehat{l_s}\cdots l_m)^{n}\right).
	\end{eqnarray*}
Rearranging:	
	\begin{eqnarray}
	P&=&x_j^{n-1}(x_{1}\cdots\widehat{x_{j}}\cdots x_{n-1}l_{n})^{n-2}(l_{n+1}\cdots l_m)^{n-1}f_{x_j}\nonumber\\
	&&-(x_{1} \cdots x_{j} \cdots x_{n-1})^{n-1}l_{n}^{n-2}(l_{n+1}\cdots l_m)^n(a_{n,j}x_{j})\nonumber\\
	&&-(x_{1} \cdots x_{j} \cdots x_{n-1}l_n)^{n-1}(l_{n+1}\cdots l_m)^{n-1}\left(\sum_{s=n+1}^{p}  a_{s,j}x_{j}l_{n+1}\cdots \widehat{l_s}\cdots l_m\right)\nonumber\\
	&&-(x_{1} \cdots \widehat{x_{j}} \cdots x_{n-1}l_n)^{n-1}x_{j}^{n}\left(\sum_{s=p+1}^{m}  a_{s,j}l_s^{n-1}(l_{n+1}\cdots \widehat{l_s}\cdots l_m)^{n}\right)\nonumber
	\end{eqnarray}
	
Summing up for $j=1,\ldots,n-1$ gives

	\begin{eqnarray}
	(n-1)P&=&\sum_{j=1}^{n-1}x_j^{n-1}(x_{1}\cdots\widehat{x_{j}}\cdots x_{n-1}l_{n})^{n-2}(l_{n+1}\cdots l_m)^{n-1}f_{x_j}\nonumber\\
	&&-(x_{1} \cdots x_{n-1})^{n-1}l_{n}^{n-2}(l_{n+1}\cdots l_m)^n\left(\sum_{j=1}^{n-1}a_{n,j}x_{j}\right)\nonumber\\
	&&-(x_{1}\cdots x_{n-1}l_n)^{n-1}(l_{n+1}\cdots l_m)^{n-1}\left(\sum_{j=1}^{n-1}\sum_{s=n+1}^{p}  a_{s,j}x_{j}l_{n+1}\cdots \widehat{l_s}\cdots l_m\right)\nonumber\\
	&&-(x_{1} \cdots \widehat{x_{j}} \cdots x_{n-1}l_n)^{n-1}\left(\sum_{j=1}^{n-1}x_{j}^{n}\left(\sum_{s=p+1}^{m}  a_{s,j}l_s^{n-1}(l_{n+1}\cdots \widehat{l_s}\cdots l_m)^{n}\right)\right).\nonumber
	\end{eqnarray}	
Setting $\sum_{j=1}^{n-1}a_{s,j}x_j=l_s$, for every $n\leq s\leq p$, the above can be rewritten as
		\begin{eqnarray}
	(n-1)P&=&\sum_{j=1}^{n-1}x_j^{n-1}(x_{1}\cdots\widehat{x_{j}}\cdots x_{n-1}l_{n})^{n-2}(l_{n+1}\cdots l_m)^{n-1}f_{x_j}-(p-n+1)P\nonumber\\
	&&-(x_{1} \cdots \widehat{x_{j}} \cdots x_{n-1}l_n)^{n-1}\left(\sum_{j=1}^{d}\sum_{s=p+1}^{m} a_{s,j}l_s^{n-1}(x_jl_{n+1}\cdots \widehat{l_s}\cdots l_m)^{n}\right),\nonumber
	\end{eqnarray}
which finally affords
	\begin{eqnarray*}
	pP&=&\sum_{j=1}^{n-1}x_j^{n-1}(x_{1}\cdots\widehat{x_{j}}\cdots x_{n-1}l_{n})^{n-2}(l_{n+1}\cdots l_m)^{n-1}f_{x_j}\\
	&-&(x_{1} \cdots \widehat{x_{j}} \cdots x_{n-1}l_n)^{n-1}\left(\sum_{j=1}^{n-1}\sum_{s=p+1}^{m} a_{s,j}l_s^{n-1}(x_jl_{n+1}\cdots \widehat{l_s}\cdots l_m)^{n}\right)\\
	&=& \sum_{j=1}^{n-1} L_1\cdots\widehat{L_j}\cdots L_n f_{x_j}\\
	&-&(x_{1} \cdots \widehat{x_{j}} \cdots x_{n-1}l_n)^{n-1}\left(\sum_{j=1}^{n-1}\sum_{s=p+1}^{m} a_{s,j}l_s^{n-1}(x_jl_{n+1}\cdots \widehat{l_s}\cdots l_m)^{n}\right)
	\end{eqnarray*}
	
Thus, we just have to take care of the summands in the second term: they are of the form	
	$$(x_1\cdots\widehat{x_j}\cdots x_{n-1}l_n l_s)^{n-1}(x_jl_{n+1}\cdots \widehat{l_s}\cdots l_m)^{n},$$
	with $1\leq j\leq n-1$ and $p+1\leq s\leq m$,
	thus obtained by trading $x_j$ and $l_s$ in
	$X:=\{x_1,\ldots,x_{n-1},l_n\}$ to get $X':=\{x_1,\ldots,\widehat{x_j},\ldots,x_{n-1},l_n,l_s\}$.
	
	Now, $\{x_1,\ldots,\widehat{x_j},\ldots,x_{n-1},l_n\}$ is linearly independent otherwise one would have a linear dependence of $n-1$ forms. Next, since $l_s$ involves effectively the variable $x_n$ it is linearly independent of $\{x_1,\ldots,\widehat{x_j},\ldots,x_{n-1},l_n\}$.
	Therefore,  ${\rm rank}(X')=n$. a case taken care of before.
\end{proof}

\subsubsection{The main theorem}

\begin{thm}\label{generic} {\rm (${\rm char}(\K)\nmid m$)}\; For an $(n-1)$-generic central hyperplane arrangement, the Jacobian ideal $J_f$ is a minimal reduction of $\mathbb{I}$ with reduction number at most $n-1$.
\end{thm}
\begin{proof} Since $\mathbb{I}$ has analytic spread $n$ and $J_f$ is minimally generated by at most $n$ elements, if it is a reduction then it is minimal.
Thus, it suffices to prove that $\mathbb I^{n}\subset J_f\mathbb I^{n-1}.$
	
As in Notation~\ref{notation_main}, $\mathbb I=\langle L_1,\ldots,L_m\rangle,$ where $L_i=l_{1}\cdots \widehat{l_i}\cdots l_m,~i=1,\ldots,m.$

The standard generators of $\mathbb I^n$ have  the form

$$P:=L_{i_1}^{r_1}L_{i_2}^{r_2}\cdots L_{i_u}^{r_u},$$ where $1\leq i_1<\cdots<i_u\leq m$, $r_1+r_2+\cdots+r_u=n$ and $r_j\geq 1$, for all $j=1,\ldots,u$. The set of the above indices $\{i_1,\ldots,i_u\}$ will be called the {\em support} of the standard generator $P$.
For any such generator the cardinality of its support is $1\leq u\leq n$.

The method of the proof  that $P$ belongs to $J_f\mathbb I^{n-1}$ is reverse induction on $u$. Thus, set $v:=n-u$, so $0\leq v\leq n-1$.

\medskip

{\sc Initial step.} $v=0$. Then $u=n$, and hence $P=L_{i_1}L_{i_2}\cdots L_{i_n}$.
Therefore, the result follows from Lemma~\ref{bottom} by an obvious symmetry of the indices.

\medskip

{\sc Inductive step.} Let $v\geq 1$, and suppose that the result is true for $v-1$, for any $(n-1)$-generic arrangement), meaning that any standard generator thereof whose support has cardinality $u+1$ is in $J_f\mathbb I^{n-1}$.

Since $v\geq 1$, then $u\leq n-1$, and at least one of the exponents $r_j$ is greater than or equal to 2. Fix an index $k\in\{1,\ldots,u\}$ such that $r_k\geq 2$.

Since the arrangement is $(n-1)$-generic and $u\leq n-1$, the linear forms $l_{i_1},\ldots,l_{i_u}$ are $\K$-linearly independent. Thus, there is a change of variables
$$y_1\leftrightarrow l_{i_1},\ldots, y_u\leftrightarrow l_{i_u}, y_{u+1}\leftrightarrow x_{u+1},\ldots, y_n\leftrightarrow x_n,$$
where $y_1,\ldots,y_n$ are new variables over $\K$.
As per the Appendix, the change preserves the hypothesis and the sought conclusion; in particular, the transformed arrangement remains $(n-1)$-generic.

By this change, $f$  transforms into $\bar{f}:=y_1\cdots y_u \ell_{u+1}\cdots \ell_m,$ where $\ell_{u+1},\ldots,\ell_m$ are linear forms in $S:=\mathbb K[y_1,\ldots,y_n]$.
Setting $\bar{L}_i:=\bar{f}/\ell_i, i=1,\ldots,m$, where $\ell_j:=y_j$, for $j=1,\ldots,u$, the transform of $P$ is
\begin{equation*}
\bar{P}:=\bar{L}_{i_1}^{r_1}\bar{L}_{i_2}^{r_2}\cdots \bar{L}_{i_u}^{r_u}=y_1^{n-r_1}y_2^{n-r_2}\cdots y_u^{n-r_u}(\ell_{u+1}\cdots \ell_m)^n,
\end{equation*}
while the transform of the ideal $\mathbb{I}$ is the ideal $\bar{\mathbb I}:=\langle \bar{L}_1,\ldots,\bar{L}_m\rangle\subset S.$

Set $\displaystyle \bar{P}=\bar{L}_{i_k}\bar{P}'$, where $\bar{P}':=\bar{L}_{i_1}^{r_1}\bar{L}_{i_2}^{r_2}\cdots \bar{L}_{i_k}^{r_k-1}\cdots \bar{L}_{i_u}^{r_u}\in \bar{\mathbb I}^{n-1}.$

Coming back to our fixed index $k\in\{1,\ldots,u\}$ above, one has
$$\bar{f}_{y_k}= \bar{L}_{i_k}+\sum_{u+1\leq j\leq m} b_{k,j}\bar{L}_j,$$ where $b_{k,j}$ is the coefficient of $y_k$ in the expression of $\ell_j$.

That is,
$$\bar{L}_{i_k}=\bar{f}_{y_k}-\sum_{j\notin\{i_1,\ldots,i_u\}} b_{k,j}\bar{L}_j.$$
Plugging into $\displaystyle \bar{P}=\bar{L}_{i_k}\bar{P}'$ yields
$$\bar{P}=\bar{f}_{y_k}\bar{P}'-\sum_{j\notin\{i_1,\ldots,i_u\}} b_{k,j}\bar{P}'\bar{L}_j$$
Clearly, $\bar{f}_{y_k}\bar{P}'\in J_{\bar{f}}\bar{\mathbb I}^{n-1}$.
At the other end, for each $j\notin\{i_1,\ldots,i_u\}$, $\bar{P}'\bar{L}_j$ is a standard generator with support $X:=\{i_1,\ldots,i_k,\ldots,i_u,j\}$ since $r_k-1\geq 1$. Thus, $|X|=u+1$.

By the inductive hypothesis as applied to the transformed arrangement, each such $\bar{P}'\bar{L}_j$ is an element of $J_{\bar{f}}\bar{\mathbb I}^{n-1}$, and hence $\bar{P}\in J_{\bar{f}}\bar{\mathbb I}^{n-1}$ too.
Reversing back to the original variables $x_1,\ldots,x_n$, as per the Appendix, we obtain that $P\in J_f\mathbb I^{n-1}$, as was to be shown.
\end{proof}

\begin{rem}
The exact reduction number above is a function of the number of components of the arrangement -- a proof is given in \cite[Proposition 2.7]{GaSiTo}. Alternatively, if one can prove that for any  indecomposable hyperplane arrangement the reduction number is exactly $n$, then Corollary~\ref{partition_general} allows to deduce that for any  hyperplane arrangement the reduction number is exactly $n-c$, where $c$ is the number of components.
Note that this corollary is not applicable within the restricted class of $(n-1)$-generic arrangements since a component of rank $r$ might fail to be $(r-1)$-generic.
As a remedy, we believe that Theorem~\ref{generic} holds regardless of generic like restrictions.
\end{rem}

\subsection{On the linear type property}

Closing this part, we state the following

\begin{conj0}\label{conj_lin_type}  {\rm (${\rm char}(\K)\nmid m$)}\;
Let $f\in R=\K[x_1,\ldots,x_n]$ denote the defining polynomial of a central hyperplane arrangement of size  $m\geq n$. Then the Jacobian ideal $J_f$ is of linear type.
\end{conj0}
Recall that an ideal $I\subset R$  is said to be {\em of linear type} if the natural surjection from its symmetric algeba to its Rees algebra is an isomorphism.

We continue to assume that ${\rm char}(\K)\nmid m$ throughout.

Note the following  result, where an ideal $\mathfrak{a}\subset R$ is said to have property $G_{\infty}$ if $\mu(\mathfrak{a}_{\wp})\leq \Ht \wp$, for every prime ideal $\wp\subset R$.

\begin{lem}\label{reduction_is_ci}
Let $f\in R=\K[x_1,\ldots,x_n]$ denote the defining polynomial of a central hyperplane arrangement of size  $m\geq n$. If either $n\leq 3$ or else the arrangement is generic then $J_f$ satisfies property $G_{\infty}$.
\end{lem}
\demo Since $J_f$ is globally generated by $n=\Ht\fm$ elements, it suffices to consider a prime $\wp\neq \fm$.
If $n=2$, $J_f$ is a complete intersection.
If $n=3$, the ideal $J_f$ is generically a complete intersection. Indeed, let $\wp$ be a minimal prime of $J_f$, necessarily of height $2$. Write $l_i=a_ix_1+b_ix_2+c_ix_3,\, i=1,\ldots,m$. Then
\begin{eqnarray}
f_{x_1}&=& a_1l_2\cdots l_m+\cdots+a_m l_1\cdots l_{m-1}\nonumber\\
f_{x_2}&=& b_1l_2\cdots l_m+\cdots+b_m l_1\cdots l_{m-1}\nonumber\\
f_{x_3}&=& c_1l_2\cdots l_m+\cdots+c_m l_1\cdots l_{m-1}\nonumber
\end{eqnarray} all belong to $\wp$. In particular, the Euler relation implies that $f=l_1l_2\cdots l_m\in\wp$, hence there exists $j\in\{1,\ldots,m\}$ with $l_j\in\wp$. But then, $a_j(f/l_j), b_j(f/l_j), c_j(f/l_j)$ also belong to $\wp$, and since $a_j, b_j, c_j$ cannot all be zero we have $f/l_j\in\wp$. Hence there exists $k\neq j$, with $l_k\in\wp$. Since  $l_j$ and $l_k$ are not proportional, then $\wp=\langle l_j,l_k\rangle$. By a change of variables, we may assume that $\wp=\langle x,y\rangle$, corresponding to the point $[0:0:1]$ of $\pp^2={\rm Proj}(\K[x,y,z])$. Therefore, the result follows from \cite[Lemma 2.4]{Sch1}.

Finally, if the arrangement is generic, then the result follows from Theorem~\ref{saturation_and_regularity} (ii) and the fact that $\mathbb{I}$ satisfies $G_n$ in the generic case (\cite[The proof of Proposition 4.1]{GaSiTo})
\qed

\begin{prop}\label{lineartype_case3}
{\rm Conjecture~\ref{conj_lin_type}} holds for $n=3$.
\end{prop}
\demo
It follows from Lemma~\ref{reduction_is_ci} since in this case $J_f$ is generated by three elements (see \cite[Proposition 3.7]{SimVas1}).
\qed

\smallskip

Recall that the arrangement $\mathcal A$ is free if and only if ${\rm Derlog}(\mathcal A)$ is a free $R$-module or, equivalently, if and only if ${\rm Syz}(J_f)$ is a free $R$-module. In particular,  a rank$3$ central hyperplane arrangement $\mathcal A$ is free if and only if the presentation ideal of the symmetric algebra of $J_f$ is minimally generated by $2$ biforms.

\begin{cor}\label{free_rees}
	Let $\mathcal{A}$ be a rank $3$ central hyperplane arrangement, with defining polynomial $f\in R=\K[x,y,z]$. Then $\mathcal A$ is free if and only if the Rees algebra $\mathcal{R}_R(J_f)$  is a complete intersection.
\end{cor}
\demo
If $\mathcal A$ is free then its syzygy matrix is of size $3\times 2$, thus implying that the symmetric algebra of $J_f$ is a complete intersection. By Proposition~\ref{lineartype_case3}, $\mathcal{R}_R(J_f)$  is a complete intersection.
	
Conversely, since $\mu(J_f)=3$, the presentation ideal of $\mathcal{R}_R(J_f)$ has codimension $2$. Since it is a complete intersection, then it is generated by a regular sequence of two minimal bihomogeneous generators.
By Proposition~\ref{lineartype_case3}, these are a set of generators of the presentation ideal of the symmetric algebra of $J_f$.
Then the syzygy matrix of the $3$-generated ideal $J_f$ must have size $3\times 2$, i.e., $\mathcal A$ is free.
\qed

\begin{rem}
 Even for $n\geq 4$ and $m=n+1$ in the generic case the question is not trivially handled. By \cite[Corollary 4.5.4]{RoseTerao}, the minimal free resolution of $R/J_f$ has the same Betti number as those of a complete intersection of $n$-forms of length $n$. In particular,  $\syz(J_f)$ is generated by the ${n\choose 2}$ degree $m-n+1=2$ reduced Koszul syzygies of the partial derivatives.
It should be possible to  show that $\mathcal{I}$ in this situation is unmixed -- note that the argument may eventually profit from the knowledge that the Fitting ideals of $J_f$ have the expected codimension bounds.
At the other end, one of the available ways to prove the linear type property, namely, proving instead that the Koszul homology modules of $J_f$ satisfy the sliding-depth conditions, is not an obvious alternative as in no order the partial derivatives seem to be a $d$-sequence.
\end{rem}

\section{Forms of higher degrees}

Let $\mathbb K$ be an infinite field, of characteristic zero or of characteristic not dividing certain critical integers. Let $f_1,\ldots,f_m\in R:=\mathbb K[x_1,\ldots,x_n], m\geq 2$, be forms of degrees $\deg(f_i)=d_i\geq 2, i=1,\ldots,m$. Let $c:=\min\{m,n\}$.
We assume throughout that $n\geq 3$ --  otherwise  most  of the assumptions become  fuzzy or else the present objective becomes irrelevant.

\begin{defn}\label{transversal_forms}\rm
The forms $f_1,\ldots,f_m$ will be called {\em nearly transversal} if they satisfy the following conditions:
\begin{itemize}
	\item[(a)] For every $1\leq i\leq m$, the hypersurface Proj$(R/(f_i))$ is smooth -- in particular, $f_i$ is irreducible and depends on all variables.
	\item[(b)] The Jacobian matrix of any subset $S_c\subset \{f_1,\ldots,f_m\}$ with $c$ elements has maximal rank modulo primes of height $n-1$ containing $S_c$.
\end{itemize}
\end{defn}
By abuse, when a form satisfies condition (a) it will be said to be smooth.

Let $F=f_1f_2\cdots f_m$, and let $J_F\subset R$ be the Jacobian ideal of $F$, i.e. $J_F=\langle F_{x_1},\ldots,F_{x_n}\rangle$, where $\displaystyle F_{x_i}:=\frac{\partial F}{\partial x_i}, i=1,\ldots,n$. Condition (a) implies that, for any $1\leq i\leq m$, the Jacobian ideal $J_{f_i}$ is a complete intersection of codimension $n$.

Note that $J_F$ is contained in the ideal generated by the $(m-1)$-fold products of $\{f_1,\ldots,f_m\}$.
In particular, $J_F$ has codimension at most $2$. But, since $F$ is reduced, its codimension is exactly $2$.
Finally, observe that, though the given forms may have different degrees $d_1,\ldots,d_m$, the ideal $J_F$ is equigenerated in degree $d_1+\cdots+d_m-1$.

The goal is to show that ${\rm depth}(R/J_F)=0$, which is equivalent to showing that $J_F$ is non-saturated  with respect to ${\frak m}:=\langle x_1,\ldots,x_n\rangle$.
Such a result would be a rough analog of Rose--Terao--Yuzvinsky theorem of Subsection~\ref{Yuz}.
The  overall expectation is to induct on the number $m\geq 2$ of forms.
Thus, the case where $m=2$ sticks up as the first to look at.

\subsection{The two forms case}

Recall that $\syz (J_F)$ denotes the syzygy module of the partial derivatives of $F$ (Subsection~\ref{indeg_of_Jacobian}.)

\begin{prop}\label{two_forms_all}
Let $f,g\in R$ be forms of respective degrees $2\leq d\leq e$.
Setting $F:=fg$, one has:
\begin{enumerate}
	\item[{\rm (i)}] If $f$ is smooth then $g^2\in J_F^{\rm sat}$, where $F=fg$.
	\item[{\rm (ii)}]  Assume that ${\rm char}(\K)$ does not divide either $e$ or $d+e$. If $f$ and $g$ have no proper common factors, then
	$${\rm indeg}(\syz (J_F))\geq {\rm indeg}(\syz (J_g)).$$
	In particular, if $g$ is smooth then ${\rm indeg}(\syz (J_F))\geq e-1\geq  \displaystyle\left\lfloor \frac{d+e}{2}\right\rfloor -1$.
	\item[{\rm (iii)}] Assume that ${\rm char}(\K)$ does not divide any of $d,e,d+e$ and suppose that the following conditions are satisfied:
	\begin{enumerate}
		\item[{\rm (a)}] $f$ is smooth.
		\item[{\rm (b)}] The Jacobian matrix of $\langle f,g\rangle$ has maximal rank over $R/\wp$, for every prime ideal $\wp\supset \langle f,g\rangle$ of height $n-1$.
		\item[{\rm (c)}] The ideal $\langle f,J_g\rangle$ is $\fm$-primary.
	\end{enumerate}
Then $\depth {R/J_F}=0.$
\end{enumerate}
\end{prop}
\demo
(i) By Leibniz,
\begin{equation}\label{partial_sums_1}
F_{x_i}=f_{x_i}\cdot g+f\cdot g_{x_i},
\end{equation}
for every $1\leq i\leq n$.
Multiplying through by $g$ yields $f_{x_i}g^2\in J_F$ for every $1\leq i\leq n$.
Since $f$ is a smooth form of degree $d\geq 2$, $J_f$ is $\fm$-primary, and therefore $g^2\in (J_F)^{\rm sat}$.

(ii) Let $(P_1\cdots P_n)^t\in \syz (J_F)$ be a nonzero syzygy.
By (\ref{partial_sums_1}), one has $$f(P_1g_{x_1}+\cdots+P_ng_{x_n})+g(P_1f_{x_1}+\cdots+P_nf_{x_n})=0.$$
Since $\gcd(f,g)=1$, then
\begin{equation}\label{derlog_relations_G}
P_1g_{x_1}+\cdots+P_ng_{x_n}=H\cdot g,
\end{equation}
for a suitable form $H$ of degree $\deg P_i-1$.
Thus, $(P_1\cdots P_n)^t\in \mathrm{Derlog}(g)= \mathrm{Syz}(J_g)\oplus R\theta_E$, where $\theta_E$ denotes the Euler vector of $g$ (see (\ref{splitting})).
Therefore, $\displaystyle (P_1-\frac{1}{e} H x_1\cdots P_n-\frac{1}{e} H x_n)^t\in \mathrm{Syz}(J_g)$.
Now, if some $P_i-\frac{1}{e} H x_i$ does not vanish, then $\deg(P_i)=\deg P_i-\frac{1}{e} H x_i$ for all $i$, and hence $\deg(P_i)\geq {\rm indeg}(\syz (J_g))$.
Else,  $P_i=\frac{1}{e} H x_i$ for all $i$.
Since $H\neq 0$ in this situation, this means that $(x_1\cdots x_n)^t\in \syz (J_F)$. But since $d+e\neq 0$ by assumption, then the Euler relation implies that $F=0$ -- an absurd.

The supplementary assertion is clear since for a complete intersection the syzygies are the Koszul relations. The appended inequality is obvious since $e=\left\lfloor \frac{2e}{2}\right\rfloor\geq  \displaystyle\left\lfloor \frac{d+e}{2}\right\rfloor$.

(iii) Let $\Theta$ denote the Jacobian matrix of $\langle f,g \rangle$.

\smallskip

{\sc Claim 1.} The homogeneous ideal $\mathbb J:=\langle f, I_2(\Theta)\rangle$ is $\fm$-primary.

Suppose otherwise and pick a prime ideal $\mathbb J\subset\wp\subset R$ properly contained in $\fm$.
Since $f\in \wp$ is assumed to be smooth, there is a partial derivative, say, $f_{x_1}$ not belonging to $\wp$.
At the other end,  $g_{x_i}f_{x_1}-g_{x_1}f_{x_i}\in \mathbb J\subset \wp$, for every $i$.
Using the Euler relation of $g$, we can write
\begin{eqnarray*}
egf_{x_1}&=&x_1f_{x_1}g_{x_1}+x_2f_{x_1}g_{x_2}+\cdots + x_nf_{x_1}g_{x_n}\\
&=& x_1f_{x_1}g_{x_1}+ [x_2(f_{x_1}g_{x_2}-f_{x_2}g_{x_1})+ x_2f_{x_2}g_{x_1}]+\cdots +[x_n(f_{x_1}g_{x_n}-f_{x_n}g_{x_1}) +x_nf_{x_n}g_{x_1}]\\
&=& g_{x_1}(x_1f_{x_1}+\cdots + x_nf_{x_n}) + x_2(f_{x_1}g_{x_2}-f_{x_2}g_{x_1}) + \cdots + x_n(f_{x_1}g_{x_n}-f_{x_n}g_{x_1})\\
&=& dg_{x_1}f + x_2(f_{x_1}g_{x_2}-f_{x_2}g_{x_1}) + \cdots + x_n(f_{x_1}g_{x_n}-f_{x_n}g_{x_1}).
\end{eqnarray*}
Clearly, the last expression belongs to $\mathbb J\subset \wp.$ Since $f_{x_1}\notin \wp$ then $g\in\wp$.
We thus conclude that $\langle f,g\rangle\subset \wp$.

Next, we take care of the dichotomy as to whether $g_{x_1}\in \wp$ or $g_{x_1}\notin \wp$.
If  $g_{x_1}\in \wp$ is the case, then $g_{x_i}\in\wp$ for every $i$, since $g_{x_i}f_{x_1}-g_{x_1}f_{x_i}\in \mathbb J\subset \wp$, and $f_{x_1}\notin \wp$. Therefore, if $g_{x_1}\in \wp$, then $\langle f,J_g\rangle\subset \wp$ as well, contradicting assumption (c).

Thus, we must have $g_{x_1}\notin \wp$.
We may assume that $\wp$ has height $n-1$.
We then contend that this contradicts condition (b) in the hypothesis.
And indeed, one has the following nonzero relation between the two rows of the Jacobian matrix of $\langle f,g\rangle$:
$$(g_{x_1}, g_{x_2},\ldots, g_{x_n})f_{x_1}-(f_{x_1}, f_{x_2},\ldots, f_{x_n})g_{x_1}=(0, f_{x_1}g_{x_2}-f_{x_2}g_{x_1}, \ldots, f_{x_1}g_{x_n}-f_{x_n}g_{x_1}),$$
which vanishes modulo $\wp$. Since neither $f_{x_1}$ nor $g_{x_1}$ vanishes modulo $\wp$, we have a contradiction.

This proves  Claim 1.

\medskip

{\sc Claim 2.}
$g\mathbb J\subset J_F$.

From (\ref{partial_sums_1}) one easily deduces the relations
$$g_{x_i}F_{x_j}-g_{x_j}F_{x_i}=(g_{x_i}f_{x_j}-g_{x_j}f_{x_i})g,$$
for $1\leq i<j\leq n$.
Therefore, $g\,I_2(\Theta)\subset J_F$.
At the other end, $gf=F\in J_F$ by the Euler relation of $F$.
Together they imply the claim.

Collecting the assertions of the two claims, one gets $g\in (J_F)^{\rm sat}$.
Clearly, $g\notin J_F$ for degree reasons, since ${\rm indeg}(J_F)=e+d-1$.
Therefore, $\fm$ is an associated prime of $R/J_F$, as was to be shown.
\qed

\begin{rem}
In the light of the proposition, perhaps it should be noted that finding reduced homogeneous free divisors looks in a sense like a sparse deal.  Accordingly, the above result tells us that if the data are in some sort of good position, the chances are dimmed.
 Note that the result generates no conflict with \cite[Theorem 2.12]{hom_divisors} because in loc. cit. the form $g$ -- that plays the role of $f$ here -- is only smooth when $d\leq 2$.
\end{rem}

\begin{quest}\rm
Note that the hypotheses of Proposition~\ref{two_forms_all} (iii) imply that $\langle f,g\rangle$ is an isolated singularity, i.e., ${\rm Proj} (R/\langle f,g\rangle)$ is smooth. Is the converse true?
\end{quest}

\subsection{Arbitrary number of forms}

We now go back to the case of forms $f_1,\ldots,f_m$ of respective degrees $d_1,\ldots,d_m$.
As before, set $F=f_1\cdots f_m$.

\begin{prop} \label{many_forms} ${\rm indeg}(\syz (J_F))\geq
	 \displaystyle\left\lfloor \frac{d_1+\cdots+d_m}{m}\right\rfloor-1$.
\end{prop}
\begin{proof}
	We prove the result by induction on $m\geq 2$.
For $m=2$, this is Proposition~\ref{two_forms_all} (ii).
		
For the inductive step, suppose $m\geq 3$.
Set $\delta:=\displaystyle\left\lfloor \frac{d_1+\cdots+d_m}{m}\right\rfloor$.
Say, $d_1$ is least among the degrees. Set $f:=f_1$, and $G:=F/f$. Let $\displaystyle\delta':=\left\lfloor \frac{d_2+\cdots+d_m}{m-1}\right\rfloor$. Clearly, $\delta'\geq \delta$ since  $(m-1)(d_1+\cdots+d_m)\leq m(d_2+\cdots+d_m)$.

 Suppose to the contrary, that $J_F$ has a syzygy $(P_1,\ldots,P_n)$ of degree $\delta-2$. By the same token as in the proof of Proposition~\ref{two_forms_all} (ii), since $\gcd(f,G)=1$ (because $\gcd(f_i,f_j)=1, i\neq j$), we have the following syzygies of degree $\delta-2$
		\begin{eqnarray}
		(P_1-(H/d)x_1)G_{x_1}+\cdots+(P_n-(H/d)x_n)G_{x_n}&=&0\nonumber\\
		(P_1+(H/d_1)x_1)f_{x_1}+\cdots+(P_n+(H/d_1)x_n)f_{x_n}&=&0,\nonumber
		\end{eqnarray} where $d:=\deg(G)$. By the inductive hypothesis, $J_G$ doesn't have any syzygy of degree $\leq \delta'-2$. Then, from the first equation we have $P_i=(H/d)x_i$, for all $i=1,\ldots,n$. This leads to $x_1f_{x_1}+\cdots+x_nf_{x_n}=0$, contradicting the Euler relation of $f\neq 0$.
\end{proof}

\medskip

\begin{prop}\label{unbalanced_degrees} Suppose there exist $j\in \{1,\ldots,m\}$ such that $$ (d_1+\cdots+d_{j-1}+d_{j+1}+\cdots+d_m)-d_j\leq \left\lfloor\frac{d_1+\cdots+d_{j-1}+d_{j+1}+\cdots+d_m}{m-1}\right\rfloor-3.$$
	Then  ${\rm depth}(R/J_F)=0$.
\end{prop}
\begin{proof} Let $f:=f_j$ and $G:=F/f$. By Proposition~\ref{two_forms_all} (i), we have $G^2\in (J_F)^{\rm sat}$. We will show now that $G^2\notin J_F$.
	
	Let $e:=d_j=\deg(f)$, and $d:=d_1+\cdots+d_{j-1}+d_{j+1}+\cdots+d_m=\deg(G)$. Suppose $G^2\in J_F$. Then there exists $L_1,\ldots,L_n\in R_{d-e+1}$, such that
	$$G^2=(L_1f_{x_1}+\cdots+L_nf_{x_n})G+f(L_1G_{x_1}+\cdots+L_nG_{x_n}).$$
	
	By grouping, we have
	
	$$G[G-(L_1f_{x_1}+\cdots+L_nf_{x_n})]=(L_1G_{x_1}+\cdots+L_nG_{x_n})f.$$
	Clearly, $\gcd(f,G)=1$ since $\gcd(f_i,f_j)=1$ for every $i\neq j$. therefore,
	
	\begin{eqnarray}
	G-(L_1f_{x_1}+\cdots+L_nf_{x_n})&=&B\cdot f\nonumber\\
	L_1G_{x_1}+\cdots+L_nG_{x_n}&=& B\cdot G,\nonumber
	\end{eqnarray} for some $B\in R_{d-e}$.
	
	Euler's relation says that $dG=x_1G_{x_1}+\cdots+x_nG_{x_n}$, which plugged into the second equation leads to the syzygy of degree $d-e+1$.
	
	$$(L_1-(B/d)x_1)G_{x_1}+\cdots+(L_n-(B/d)x_n)G_{x_n}=0.$$
	
	Since from hypotheses $d-e\leq \lfloor d/(m-1)\rfloor-3$, from Proposition~\ref{many_forms}, we have that $L_i=(B/d)x_i$ for all $i=1,\ldots,n$. This plugged in the first equation leads to $$G=Bf(e/d-1);$$ a contradiction. \end{proof}

\section{Appendix}

\subsection{Calculus on Jacobian ideals}

In order to ease the amount of calculations needed in our proofs, we often appeal to change of variables/coordinates technique to simplify the shape of the defining polynomial of our hyperplane arrangement. The chain rule of partial differentiation and linear algebra will safely allow us to do this without altering the homological properties of the Jacobian ideal we are studying throughout these notes.

Concretely, let $F=l_1\cdots l_nl_{n+1}\cdots l_m\in R:=\mathbb K[x_1,\ldots,x_n]$, where $l_{i_1},\ldots,l_{i_n}$ are linearly independent linear forms in $R$, for some distinct indices $i_1,\ldots,i_n\in\{1,\ldots,m\}$. After a relabeling, we can assume that $i_1=1,\ldots,i_n=n$. Then there is a unique change of variables $l_1\leftrightarrow y_1,\ldots, l_n\leftrightarrow y_n$, given by $$[x_1,\ldots,x_n]\cdot M=[y_1,\ldots,y_n],$$ where $M$ is the $n\times n$ invertible matrix whose columns are the coefficients of the linear forms $l_1,\ldots,l_n$, respectively.

Then, we have $[x_1,\ldots,x_n]=[y_1,\ldots,y_n]\cdot M^{-1}$, and hence $x_1=L_1,\ldots, x_n=L_n$, where $L_1,\ldots,L_n$ are linear forms in $S:=\mathbb K[y_1,\ldots,y_n]$. Then we consider

$$G:=F(L_1,\ldots,L_n)=y_1\cdots y_n(l_{n+1}'\cdots l_m')\in S=\mathbb K[y_1,\ldots,y_n].$$

The chain rule says that for each $i=1,\ldots,n$, in $\mathbb K[y_1,\ldots,y_n]$ we have

$$\frac{\partial G}{\partial y_i}=\frac{\partial F}{\partial x_1}(L_1,\ldots,L_n)\cdot \frac{\partial L_1}{\partial y_i}+\cdots +\frac{\partial F}{\partial x_n}(L_1,\ldots,L_n)\cdot \frac{\partial L_n}{\partial y_i},$$ or, as vectors with entries elements of $S$, we can write

$$[G_{y_1},\ldots,G_{y_n}]=[F_{x_1}(L_1,\ldots,L_n),\ldots,F_{x_n}(L_1,\ldots,L_n)]\cdot (M^{-1})^\tau,$$ where $\tau$ denotes the transposition of the matrix. Since the matrix $(M^{-1})^\tau$ is invertible, in terms of ideals of $S=\mathbb K[y_1,\ldots,y_n]$, we have

$$\langle G_{y_1},\ldots,G_{y_n}\rangle=\langle F_{x_1}(L_1,\ldots,L_n),\ldots,F_{x_n}(L_1,\ldots,L_n)\rangle.$$

Step-by-step, this is the path we are following:

\begin{itemize}
  \item[(1)] We start with $J_F=\langle F_{x_1},\ldots, F_{x_n}\rangle$ in $R=\mathbb K[x_1,\ldots,x_n]$.
  \item[(2)] We make the change of variables (from $x_i$'s to $y_j$'s) $x_1\leftrightarrow L_1,\ldots,x_n\leftrightarrow L_n$, to obtain an ideal $\bar{J}$ in $S=\mathbb K[y_1,\ldots,y_n]$ that is generated by $F_{x_1}(L_1,\ldots,L_n),\ldots,F_{x_n}(L_1,\ldots,L_n)$.
  \item[(3)] The ideal $\bar{J}$ of $S$ obtained at (2) has the same homological properties as the ideal $J_F$ of $R$.
  \item[(4)] But, as we observed, $\bar{J}=J_G\subset S$. Therefore, in the shape of $F$, we can assume that $l_1=x_1,\ldots,l_n=x_n$, as we did throughout this paper.
\end{itemize}

\medskip

This procedure works also when we want to show inclusions of the form $\mathbb I^k\subseteq J_F \mathbb I^{k-1}$, for some $k$.

Denote $\bar{\mathbb I}$ to be the ideal of $S=\mathbb K[y_1,\ldots,y_n]$ generated by all $(m-1)$-fold products of linear forms $y_1,\ldots,y_n,l_{n+1}',\ldots,l_m'$. Suppose we showed $$\bar{\mathbb I}^k\subseteq J_G\bar{\mathbb I}^{k-1}$$ in $S=\mathbb K[y_1,\ldots,y_n]$.

If we go back to variables $x_1,\ldots,x_n$, via the change of variables $y_1\leftrightarrow l_1,\ldots,y_n\leftrightarrow l_n$, since $J_G$ will transform into $J_F$ (by making this substitution in the chain rule above, and because $M$ is invertible), and since $\bar{\mathbb I}$ will transform back into $\mathbb I$, we will get also that $\mathbb I^k\subseteq J_F \mathbb I^{k-1}$.

\medskip

\begin{exm} Suppose $n=2$, and say, $$F=\underbrace{(x_1+x_2)}_{l_1}\underbrace{(2x_1+x_2)}_{l_2}\underbrace{(x_1-x_2)}_{l_3} \underbrace{(x_1+3x_2)}_{l_4}\in R:=\mathbb K[x_1,x_2].$$ Then we have:

$$[x_1, x_2] \underbrace{\left[\begin{matrix}
1 & 2\\
1 & 1
\end{matrix}
\right]}_{M}=[y_1, y_2].
$$ Then, $$[x_1, x_2]=[y_1,y_2] \underbrace{\left[\begin{matrix}
-1 & 2\\
1 & -1
\end{matrix}
\right]}_{M^{-1}}=[\underbrace{-y_1+y_2}_{L_1},\underbrace{2y_1-y_2}_{L_2}],
$$ and so $$G=y_1y_2(-3y_1+2y_2)(5y_1-2y_2)\in S=\mathbb K[y_1,y_2].$$

We have

$$F_{x_1}=8x_1^3+21x_1^2x_2+2x_1x_2^2-7x_2^3 \mbox{ and }F_{x_2}=7x_1^3+2x_1^2x_2-21x_1x_2^2-12x_2^3,$$ as well as

$$G_{y_1}=-45y_1^2y_2+32y_1y_2^2-4y_2^3\mbox{ and }G_{y_2}=-15y_1^3+32y_1^2y_2-12y_1y_2^2.$$

Then we have:

\begin{eqnarray}
F_{x_1}(L_1,L_2)&=&-30y_1^3+19y_1^2y_2+8y_1y_2^2-4y_2^3\nonumber\\
F_{x_2}(L_1,L_2)&=&-15y_1^3-13y_1^2y_2+20y_1y_2^2-4y_2^3,\nonumber
\end{eqnarray} which indeed satisfy

$$[G_{y_1}, G_{y_2}]=[F_{x_1}(L_1,L_2), F_{x_2}(L_1,L_2)] \underbrace{\left[\begin{matrix}
-1 & 1\\
2 & -1
\end{matrix}
\right]}_{(M^{-1})^{\tau}}.$$
\end{exm}

\vskip 0.3in

\noindent {\bf Acknowledgment.} Ricardo Burity thanks the Department of Mathematics of the University of Idaho (USA) for the hospitality during his stay. He is grateful as well to CAPES (MEC, Brazil) for funding the current Post-Doctoral scholarship at the University of Idaho.

\bigskip

\bibliographystyle{amsalpha}

\end{document}